\documentclass[11pt]{amsart}

\usepackage{amsmath,amssymb,esint}
\usepackage{amsthm}
\usepackage{hyperref}
\usepackage[margin=1.0in]{geometry}

\numberwithin{equation}{section}

\newtheorem{prop}{Proposition}
\newtheorem{lemma}[prop]{Lemma}
\newtheorem{thm}[prop]{Theorem}
\newtheorem{cor}[prop]{Corollary}

\numberwithin{prop}{section}

\theoremstyle{definition}
\newtheorem{defn}[prop]{Definition}

\newtheorem{rmk}[prop]{Remark}

\newcommand{\del}{\partial}
\newcommand{\delb}{\bar{\partial}}\newcommand{\dt}{\frac{\partial}{\partial t}}
\newcommand{\brs}[1]{\left| #1 \right|}

\renewcommand{\gg}{\gamma}
\newcommand{\gD}{\Delta}
\newcommand{\gd}{\delta}
\newcommand{\gs}{\sigma}

\newcommand{\gL}{\Lambda}
\newcommand{\gk}{\kappa}
\newcommand{\gl}{\lambda}

\newcommand{\gw}{\omega}
\newcommand{\ga}{\alpha}
\newcommand{\gb}{\beta}
\renewcommand{\ge}{\epsilon}
\newcommand{\N}{\nabla}
\newcommand{\FF}{\mathcal F}
\newcommand{\GG}{\mathcal G}
\newcommand{\WW}{\mathcal W}

\newcommand{\til}[1]{\widetilde{#1}}

\renewcommand{\bar}[1]{\overline{#1}}

\renewcommand{\i}{\sqrt{-1}}

\newcommand{\IP}[1]{\left<#1\right>}

\DeclareMathOperator{\Rc}{Rc}
\DeclareMathOperator{\Rm}{Rm}

\DeclareMathOperator{\tr}{tr}

\DeclareMathOperator{\divg}{div}

\DeclareMathOperator{\Vol}{Vol}
\DeclareMathOperator{\diam}{diam}

\DeclareMathOperator{\Area}{Area}

\begin{document}

\title[Ricci-Yang-Mills flow on surfaces and pluriclosed flow on elliptic fibrations]{Ricci-Yang-Mills flow on surfaces and pluriclosed flow on elliptic fibrations}

\begin{abstract} We give a complete description of the global existence and convergence for the Ricci-Yang-Mills flow on $T^k$ bundles over Riemann surfaces.  These results equivalently describe solutions to generalized Ricci flow and pluriclosed flow with symmetry.
\end{abstract}

\author{Jeffrey Streets}
\address{Rowland Hall\\
         University of California, Irvine\\
         Irvine, CA 92617}
\email{\href{mailto:jstreets@uci.edu}{jstreets@uci.edu}}

\thanks{We gratefully acknowledge support from the NSF via DMS-1454854}

\date{\today}

\maketitle

\section{Introduction}

Suppose $\mathcal G \to M \to \Sigma$ is the total space of a principal $\mathcal G$-bundle.  Given a choice $h$ of metric on the Lie algebra $\mathfrak g$ of $\GG$, a one-parameter family of metrics $g_t$ on $\Sigma$ and principal connections $\mu_t$ satisfies the \emph{Ricci-Yang-Mills flow (RYM flow)} \cite{StreetsThesis, YoungThesis} if
\begin{gather} \label{f:RYM}
\begin{split}
\dt g =&\ -2 \Rc_g + F_{\mu}^2,\\
\dt \mu =&\ - d^*_{g} F_{\mu},
\end{split}
\end{gather}
where $F_{\mu}$ denotes the curvature of $\mu$, and $F_{\mu}^2 = \tr_h \tr_{g^{-1}} F_{\mu} \otimes F_{\mu}$.
This system of equations has arisen in mathematical physics literature in the study of renormalization group flows.  Also, this flow arises by modifying the Ricci flow of an invariant metric on a principal bundle to fix the metric on the fibers.  This modification is a natural way to simplify the equation to understand the geometry of principal bundles, while of course the Ricci flow of invariant metrics is also natural in understanding collapsing limits and has been significantly studied (cf. \cite{LottDR}).  Further results on RYM flow appear in \cite{JYRYMsol,StreetsRGflow,StreetsRYMsurfaces,YoungSRYM}.

A second origin of these equations is as a symmetry reduction of the \emph{generalized Ricci flow} (cf. \cite{GRFbook}).  Given a smooth manifold $M$, a one-parameter family of metrics $G_t$ and closed three-forms $H_t$ is a solution of generalized Ricci flow if
\begin{gather} \label{f:GRF}
\begin{split}
\dt G =&\ -2 \Rc_G + \tfrac{1}{2} H^2,\\
\dt H =&\ \gD_d H,
\end{split}
\end{gather}
where $H^2 = \tr_{g^{-1}} \tr_{g^{-1}}H \otimes H$.
This system of equations also arises in the study of renormalization group flows \cite{Friedanetal}, coupling the Ricci flow for a metric with the natural heat flow for a closed three-form.  As discussed in \S \ref{s:setup} below, given a solution to RYM flow $(g_t, \mu_t)$ over a Riemann surface, the pairs
\begin{align*}
G_t = \pi^* g_t + \tr_h \mu_t \otimes \mu_t, \qquad H_t = \tr_h F_{\mu} \wedge \mu
\end{align*}
satisfy generalized Ricci flow.

A third origin of Ricci-Yang-Mills flow is complex geometry, where it arises as a special case of \emph{pluriclosed flow} \cite{PCF}.  The pluriclosed flow is a natural evolution equation generalizing the K\"ahler-Ricci flow to complex, non-K\"ahler manifolds.  Given $(M^4, J)$ a complex surface, a one-parameter family of Hermitian, pluriclosed metrics $\gw_t$ satisfies pluriclosed flow if
\begin{align*}
\dt \gw =&\ - \rho_B^{1,1},
\end{align*}  
where $\rho_B$ is the Ricci curvature of the Bismut connection.  In \cite{PCFReg} it was shown that after a certain gauge modification pluriclosed flow is in fact equivalent to generalized Ricci flow.  Furthermore, it was shown in \cite{Streetssolitons} that the pluriclosed flow of $T^2$-invariant metrics on complex surfaces reduces to the RYM flow.  What is somewhat surprising is that the pluriclosed flow, defined in general with no symmetry considerations in mind, naturally fixes the metric on the Lie algebra and results in the RYM flow in this setting.

In this paper, building on prior results (cf. \cite{StreetsRYMsurfaces}), we give a complete description of the Ricci-Yang-Mills flow for $T^k$-bundles over Riemann surfaces.  A complete description of the Ricci flow on Riemann surfaces was achieved by Hamilton/Chow \cite{ChowRFSurfaces, HamiltonRFSurfaces}, with other approaches coming later (cf. \cite{BartzRF, StruweFlows}).  Also, a complete description of Yang-Mills flow over Riemann surfaces was given by R\aa de \cite{Rade}.  We first state the main result describing the behavior of Ricci-Yang-Mills flow:

\begin{thm} \label{t:RYMthm} Let $T^k \to M \to \Sigma$ denote a principal $T^k$-bundle over a Riemann surface $\Sigma$, and fix $h$ a choice of metric on $\mathfrak{t}^k$.  Let $g_0$ denote a Riemannian metric on $\Sigma$ and $\mu_0$ a choice of principal connection.  We let $(g_t, \mu_t)$ denote the solution to Ricci-Yang-Mills flow with this initial data, and $G_t = \pi^* g_t + \tr_h \mu_t \otimes \mu_t$ the associated one-parameter family of invariant metrics on $M$.  The following hold:
\begin{enumerate}
\item If $\chi(\Sigma) < 0$ then $(g_t, \mu_t)$ exists on $[0,\infty)$ and $(M, \frac{G_t}{2t})$ converges in the Gromov-Hausdorff topology to $(\Sigma, g_{\Sigma})$, where $g_{\Sigma}$ denotes the canonical metric of constant curvature $-1$.
\item If $\chi(\Sigma) = 0$ then $(g_t, \mu_t)$ exists on $[0,\infty)$ and $(M, \frac{G_t}{2t})$ converges in the Gromov-Hausdorff topology to a point.
\item If $\chi(\Sigma) > 0$ and $c_1(M) = 0$, then there exists $T < \infty$ such that $(g_t, \mu_t)$ exists on $[0,T)$, and $(M, \frac{1}{T - 2t}G_t)$ converges in the $C^{\infty}$ topology to $(\Sigma \times \mathbb R, g_{\Sigma} \times g_{\mathbb R})$, where $g_{\Sigma}$ denotes a metric of constant curvature $1$.
\item If $\chi(\Sigma) > 0$, and $c_1(M) \neq 0$, then $(g_t, \mu_t)$ exists on $[0,\infty)$ and there are constants $\gl_i$ and a one-parameter family of diffeomorphisms $\phi_t$ such that
\begin{align*}
\lim_{t \to \infty} \phi_t^* g_t =&\ \gl_1 g_{\Sigma}\\
\lim_{t \to \infty} \phi_t^* F_{\mu_t} =&\ \gl_2 \gw_{\Sigma},
\end{align*}
where $g_{\Sigma}$ denotes a metric of constant curvature $1$, and $\gw_{\Sigma}$ denotes the associated area form.
\end{enumerate}
\end{thm}

\begin{rmk} Note that the qualitative behavior of the metric $g_t$ in cases (1) and (2) agrees with the case of Ricci flow, regardless of the topology of the bundle.  The cases (3) and (4), where $\chi(\Sigma) > 0$, are more subtle.  We recall that the Ricci flow on $S^2$ with arbitrary initial data contracts in finite time to a round point \cite{ChowRFSurfaces, HamiltonRFSurfaces}.  In case the bundle is trivial, the RYM flow includes these Ricci flow lines as special cases, and case (3) shows that this behavior still holds in general, that is, the flow still converges to a round point, with $F_{\mu} \to 0$ in the limit.

Considering homogeneous examples, one sees a basic qualitative difference in the case when the bundle is nontrivial, that is, when $F_{\mu} \neq 0$.   Specifically, the $F_{\mu}^2$ term in the evolution of $g$ acts as a \emph{restoring force} which doesn't allow the sphere to collapse.  Rather, along RYM flow the base will converge to a round sphere of fixed size depending on the Chern class of the bundle and the choice of $h$, without need for normalization, with $F_{\mu}$ remaining fixed (note this indicates the necessity of the constants $\gl_i$).  Outside of the homogeneous setting one can hope that this behavior will still hold due to the topological nontriviality of the bundle.  Case (4) of Theorem \ref{t:RYMthm} confirms this behavior for all initial data.  
\end{rmk}

\begin{rmk} The proofs of cases (1) and (2) follow from maximum principle arguments combined with monotonicity of a modified Liouville energy.   On the other hand, the proofs of cases (3) and (4) are significantly more intricate.  For case (3) we rely on a modified Perelman-type entropy functional to prove a $\gk$-noncollapsing result for the metrics along the flow.  Constructing a blowup limit at a singular time thus yields an ancient solution, which by maximum principle arguments can be shown to satisfy $F_{\mu} \equiv 0$, so that it is in fact a solution to Ricci flow.  A result of Perelman \cite{Perelman1} yields that this is then isometric to the shrinking sphere solution, yielding the claimed behavior.  Case (4) requires studying a certain gauge-modified flow to allow for an application of Aubin's improved Moser-Trudinger inequality.  We furthermore use the topological hypothesis of nontriviality of the bundle to establish an a priori lower bound for the volume along the flow.
\end{rmk}

As discussed above, Ricci-Yang-Mills flow over Riemann surfaces is equivalently described by solutions to generalized Ricci flow, thus Theorem \ref{t:RYMthm} has an immediate corollary:

\begin{cor} \label{c:GRFcor} Let $T^k \to M \to \Sigma$ denote a principal $T^k$ bundle over a Riemann surface $\Sigma$, with $h$ a choice of metric on $\mathfrak{t}^k$.  Given $G_0 =  \pi^* g_0 + \tr_h \mu_0 \otimes \mu_0$ an invariant metric on $M$, let $H_0 = \tr_{h} F_{\mu_0} \wedge \mu_0$.  Let $(G_t, H_t)$ denote the unique solution to generalized Ricci flow on $M$ with initial data $(G_0, H_0)$.  Then $G_t =  \pi^* g_t + \tr_h \mu_t \otimes \mu_t, H_t = \tr_h F_{\mu_t} \wedge \mu_t$, where $(g_t, \mu_t)$ is the solution to RYM flow with initial condition $(g_0, \mu_0)$.  In particular, the existence and convergence properties are described as in Theorem \ref{t:RYMthm} according to the topology of $M$.
\end{cor}

\begin{rmk} \label{r:Tdualityrmk} As solutions to generalized Ricci flow with a torus symmetry, the flow lines in Corollary \ref{c:GRFcor} are subject to the action of $T$-duality (cf. \cite{GF19, GRFbook, StreetsTdual}).  As an elementary example, homogeneous solutions on the unit tangent bundle over $\Sigma$ are $T$-dual to solutions on the same bundle with $g$ and $\mu$ preserved, while the length of the circles is inverted.
\end{rmk}

Furthermore, due to the connection to pluriclosed flow described above, Theorem \ref{t:RYMthm} has applications to complex geometry.  Complex surfaces give the first examples of non-K\"ahler manifolds and among these, elliptic surfaces form a diverse and interesting class.  Conjectures of the behavior of pluriclosed flow on these surfaces were described in \cite{SG}, and the next result confirms these conjectures in the case of invariant initial data on principal bundles, announced in \cite{SG}.

\begin{cor} \label{t:PCFcorollary} Let $(M, J)$ be a compact complex surface which is the total space of a holomorphic $T^2$-principal bundle over a Riemann surface $\Sigma$.
\begin{enumerate}
\item Suppose $\chi(\Sigma) < 0$.  Given $\gw_0$ an invariant pluriclosed metric on $(M, J)$, the solution to pluriclosed flow with initial data $\gw_0$ exists on $[0,\infty)$, and $(M, \frac{\gw_t}{2t})$ converges in the Gromov-Hausdorff topology to $(\Sigma, g_{\Sigma})$, where $g_{\Sigma}$ denotes the canonical metric of constant curvature $-1$.
\item Suppose $\chi(\Sigma) = 0$.  Given $\gw_0$ an invariant pluriclosed metric on $(M, J)$, the solution to pluriclosed flow with initial data $\gw_0$ exists on $[0,\infty)$, and $(M, \frac{\gw_t}{2t})$ converges in the Gromov-Hausdorff topology to a point.
\item Suppose $(M, J) \cong S^2 \times T^2$.  Given $\gw_0$ an invariant pluriclosed metric on $(M, J)$, let
\begin{align*}
T = (8 \pi)^{-1} \Area_{\gw_0} (S^2).
\end{align*}
The solution to pluriclosed flow with this initial data exists on $[0,T)$, and $(M, \frac{1}{T - 2t}\gw_t)$ converges in the $C^{\infty}$ topology to $(S^2 \times \mathbb R^2, \gw_{S^2} \times \gw_{\mathbb R^2})$.
\item Suppose $(M, J)$ is a Hopf surface.  Given $\gw_0$ an invariant pluriclosed metric on $(M, J)$, the solution to pluriclosed flow with this initial data exists on $[0,\infty)$, and $(M, \gw_t)$ converges in the $C^{\infty}$ topology to a multiple of $\gw_{\mbox{\tiny{Hopf}}}$, the standard Hopf metric.
\end{enumerate}
\end{cor}

\begin{rmk} \label{r:corrmk}
\begin{enumerate}
\item This result confirms the natural maximal existence time conjecture for pluriclosed flow in this setting (\cite{PCFReg} Conjecture 5.2)
\item Up to finite covers, all non-K\"ahler surfaces of Kodaira dimension $0$ or $1$ are total spaces of holomorphic $T^2$ bundles, and admit invariant pluriclosed metrics and flow lines as described in items (1) and (2).
\item The surfaces in case (1) have Kodaira dimension $1$.  The general behavior of the K\"ahler-Ricci flow on K\"ahler surfaces of Kodaira dimension $1$, analyzing the more subtle case of singular fibrations, is described in \cite{SongTian}.
\item The surfaces in case (2) have Kodaira dimension $0$, and include both K\"ahler and non-K\"ahler Kodaira surfaces.
\item The surfaces in cases (3) and (4) have Kodaira dimension $-\infty$.  In fact, the only Hopf surfaces which are principal holomorphic $T^2$ bundles are \emph{standard}, described as $\mathbb Z$-quotients
\begin{align*}
S^3 \times S^1 \cong \mathbb C^2 \backslash \{\mathbf{0}\} / \IP{ (z_1, z_2) \to (\ga z_1, \gb z_2)},
\end{align*}
where $\brs{\ga} = \brs{\gb} < 1$.  These surfaces admit the Hermitian, pluriclosed, metric defined by the $\mathbb Z$-invariant K\"ahler form
\begin{align*}
\gw_{\mbox{\tiny{Hopf}}} = \rho^{-2} \i \del \delb \rho^2,
\end{align*}
where $\rho = \sqrt{ \brs{z_1}^2 + \brs{z_2}^2}$ is distance to the origin.  The associated Riemannian metric is the standard cylindrical metric on $\mathbb C^2 \backslash \{\mathbf{0}\} \cong S^3 \times \mathbb R$.  We denote the metric on the quotient as $\gw_{\mbox{\tiny{Hopf}}}$, which is Bismut-flat.
\end{enumerate}
\end{rmk}

Here is an outline of the rest of this paper.  In \S \ref{background} we derive fundamental properties of RYM flow on surfaces, including a reduction of the metric component to a flow of conformal factors, a scalar reduction of the Yang-Mills flow component, a monotonicity formula for a modified Liouville energy, and the derivation of some evolution equations.  We also recall a higher regularity estimate for the flow assuming certain uniform bounds on the base metric $g$.  We also describe the symmetry reductions discussed above.  Next in \S \ref{s:estimates} we record certain a priori estimates for the metric and Yang-Mills potential along the flow, adapted to each topological setting.  In \S \ref{s:convergence} we give the proofs of the main theorems, case by case.  The cases when $\chi \leq 0$ are handled using the a priori estimates of \S \ref{s:estimates}.  The case of $\chi(\Sigma) > 0$ and trivial bundle is handled by blowup argument, relying on a modification of Perelman's entropy functional to this setting.  Finally we address the case of $\chi(\Sigma) > 0$ and nontrivial bundle, relying on an a priori lower volume estimate and an estimate for the associated Sobolev constant.

\vskip 0.1in 

\noindent \textbf{Acknowledgements:} We thank Matthew Gursky for several helpful conversations, and Xilun Li for a careful reading of the paper.

\section{Background} \label{background}
\subsection{Setup and scalar reductions} \label{s:setup}

Fix $T^k \to M \to \Sigma$ a principal $T^k$ bundle over a compact Riemann surface $\Sigma$.  An invariant metric $G$ on $M$ is determined by a triple $(g, \mu, h)$, where $g$ is a metric on $\Sigma$, $\mu$ is a principal $T^k$-connection, and $h$ is a family of metrics on $\mathfrak t^k$ parameterized by $\Sigma$.  Specifically, $G = \pi^* g + \tr_h \mu \otimes \mu$.  We will make the further restriction that this family $h$ is constant over $M$, thus determined by a choice of inner product on $\mathfrak t^k$.  We will study a certain normalization of the Ricci-Yang-Mills flow system, namely we fix $\gl \in \{-1,0,1\}$ and consider
\begin{gather} \label{f:nRYM}
\begin{split}
\dt g =&\ -2 \Rc_g + F_{\mu}^2 - \gl g,\\
\dt \mu =&\ - d^*_{g} F_{\mu},\\
\dt h =&\ - \gl h.
\end{split}
\end{gather}
where $(F^2_{\mu})_{ij} = h_{IJ} g^{kl} F^I_{ik} F^J_{jl}$.  We note that our normalized flow also scales the metric $h$.  This is natural from the point of view of RYM flow as a symmetry reduction of flows on the total space of the bundle, and in particular it follows that the data described by (\ref{f:nRYM}) differs from the system by a rescaling of the associated metrics $G_t = \pi^* g_t + \tr_{h_t} \mu_t \otimes \mu_t$.  To begin our analysis we first show that RYM flow over a Riemann surface can be described using a conformal factor on the base space and a $\mathfrak t^k$-valued function determining the principal connection.

\begin{lemma} \label{l:uevolution} Given a solution to Ricci-Yang-Mills flow, one has $g_t = e^{u_t} g_{\Sigma}$, where
\begin{align} \label{f:ured}
\dt u =&\ e^{-u} \left( \gD_{g_{\Sigma}} u - R_{\Sigma} \right) + \tfrac{1}{2} \brs{F_{\mu}}^2_{g,h} - \gl.
\end{align}
\begin{proof} As $\Sigma$ is a surface, it follows easily that
\begin{align*}
\Rc_g = \tfrac{1}{2} R g, \qquad F_{\mu}^2 = \tfrac{1}{2} \brs{F_\mu}^2_{g,h} g.
\end{align*}
It follows easily that the ansatz $g_t = e^{u_t}$ is preserved, and furthermore
\begin{align*}
\left(\dt u \right) g_t = \dt g_t = - \left( R - \tfrac{1}{2} \brs{F_{\mu}}^2_{g,h} + \gl \right) g_t.
\end{align*}
Using the formula $R = e^{-u} \left( R_{\Sigma} - \gD_{g_{\Sigma}} u \right)$, the result follows.
\end{proof}
\end{lemma}

\begin{lemma} \label{l:freduction} Given a solution to Ricci-Yang-Mills flow , and a background connection $\bar{\mu}$ there exists a one-parameter family of $\mathfrak t^k$-valued functions $f_t$ such that
\begin{align} \label{f:fred}
\mu_t = \bar{\mu} + d^c f, \qquad \dt f =&\ \gD_g f + \tr_{\gw} F_{\bar{\mu}}.
\end{align}
\begin{proof} Since $\mu_0$ and $\bar{\mu}$ are connections on the same bundle over a Riemann surface, using Hodge theory and a gauge transformation we can solve for a function $f_0$ such that $\mu_0 = \bar{\mu} + d^c f_0$.  Given the solution to Ricci-Yang-Mills flow, we can solve for $f_t$ as in the statement by the theory of linear parabolic equations, using initial data $f_0$.  We then define $\til{\mu}_t = \bar{\mu} + d^c f_t$ and observe that
\begin{align*}
\dt \til{\mu} =&\ d^c \left( \gD_g f + \tr_{\gw} F_{\bar{\mu}} \right)\\
=&\ d^c \tr_{\gw} \left( F_{\bar{\mu}} + d d^c f\right)\\
=&\ d^c \tr_{\gw} F_{\til{\mu}}\\
=&\ -d^*_g F_{\til{\mu}}.
\end{align*}
Thus $\til{\mu}$ satisfies the Yang-Mills flow with respect to the time-dependent metric $g$, and since $\til{\mu}_0 = \mu_0$ and solutions to Yang-Mills flow are unique, it follows that $\mu_t = \til{\mu}_t = \bar{\mu} + d^c f$, as required.
\end{proof}
\end{lemma}

Thus we have shown that the metric $G_t$ is defined equivalently in terms of a conformal factor $u$ on $\Sigma$ and a $\mathfrak t^k$-valued function $f$, and this notation will be used throughout.  Furthermore we will refer to the principal connection as $\mu_f := \bar{\mu} + d^c f$, where $\bar{\mu}$ is some background connection.

\subsection{Energy functional}

We next exhibit a gradient formulation for Ricci-Yang-Mills flow in this setting.  From \cite{StruweFlows} we know that the Ricci flow on surfaces is the gradient flow of the Liouville energy.  A generalization of this to Ricci-Yang-Mills flow was shown in \cite{StreetsRYMsurfaces}, and below we give a minor modification of this to account for the scaling parameter $\gl$.

\begin{defn} Given data $(u,f,h)$ as above, let
\begin{align*}
\FF(u,f,h) =&\ \int_{\Sigma} \left( \tfrac{1}{2} \brs{d u}^2 + e^{-u} \brs{F_{\mu_f}}^2_{g_{\Sigma},h} \right) dV_{\Sigma} + R_{\Sigma} \int_{\Sigma} u dV_{\Sigma} + \gl \int_{\Sigma} e^u dV_{\Sigma}.
\end{align*}
\end{defn}

\begin{prop} \label{p:Liouvillemon} Given a solution to Ricci-Yang-Mills flow, one has
\begin{align*}
\frac{d}{dt} \FF(u_t,f_t,h_t) =&\ - \int_{\Sigma} e^u \dot{u}^2 dV_{\Sigma} - \gl \int_{\Sigma} e^{-u} \brs{F_{\mu_f}}^2_{g_{\Sigma},h} dV_{\Sigma} - 2 \int_{\Sigma} e^{-2u} \IP{\N^g F_{\mu_f}, \N^g F_{\mu_f}}_{g_{\Sigma},h}dV_{\Sigma}.
\end{align*}
\begin{proof} First we compute
\begin{align*}
\frac{d}{dt} \int_{\Sigma} \tfrac{1}{2} \brs{d u}^2 dV_{\Sigma} =&\ \int_{\Sigma} \IP{ d \dot{u} , du} dV_{\Sigma} = - \int_{\Sigma} \dot{u} \gD_{g_{\Sigma}} u dV_{\Sigma}
\end{align*}
Next
\begin{align*}
\frac{d}{dt} \int_{\Sigma} e^{- u} \brs{F_{\mu_f}}^2_{g_{\Sigma},h} dV_{\Sigma} =&\ \int_{\Sigma} \left[ \left( - \dot{u} - \gl \right) e^{-u} \brs{F_{\mu_f}}^2_{g_{\Sigma},h} + 2 e^{-u} \IP{\gD_g F_{\mu_f}, F_{\mu_f}}_{g_{\Sigma},h} \right] dV_{\Sigma}\\
=&\ \int_{\Sigma} \left[ \left( \dot{u} + \gl \right) \left( - e^{-u} \brs{F_{\mu_f}}^2_{g_{\Sigma},h} \right) - 2 e^{-2u} \IP{\N^g F_{\mu_f}, \N^g F_{\mu_f}}_{g_{\Sigma},h} \right] dV_{\Sigma}.
\end{align*}
Lastly
\begin{align*}
\frac{d}{dt} \left( R_{\Sigma} \int_{\Sigma} u dV_{\Sigma} + \gl \int_{\Sigma} e^u dV_{\Sigma} \right) = \int_{\Sigma} \left( R_{\Sigma} \dot{u}  + \gl \dot{u} e^u \right)dV_{\Sigma}.
\end{align*}
Combining these yields
\begin{align*}
\frac{d}{dt} \FF(u_t,f_t) =&\ \int_{\Sigma} \dot{u} \left( - \gD_{g_{\Sigma}} u + R_{\Sigma} - e^{-u} \brs{F_{\mu_f}}^2_{g_{\Sigma},h} + \gl e^u \right) dV_{\Sigma}\\
&\ - \gl \int_{\Sigma} e^{-u} \brs{F_{\mu_f}}^2_{g_{\Sigma},h} dV_{\Sigma} - 2 \int_{\Sigma} e^{-2u} \IP{\N^g F_{\mu_f}, \N^g F_{\mu_f}}_{g_{\Sigma},h}dV_{\Sigma}\\
=&\ - \int_{\Sigma} e^u \dot{u}^2 dV_{\Sigma} - \gl \int_{\Sigma} e^{-u} \brs{F_{\mu_f}}^2_{g_{\Sigma},h} dV_{\Sigma} - 2 \int_{\Sigma} e^{-2u} \IP{\N^g F_{\mu_f}, \N^g F_{\mu_f}}_{g_{\Sigma},h}dV_{\Sigma},
\end{align*}
as claimed.
\end{proof}
\end{prop}

\subsection{Higher Regularity} \label{s:higherreg}

One key application of the energy monotonicity of Proposition \ref{p:Liouvillemon} is to obtain higher regularity estimates for the flow in the presence of certain bounds.

\begin{prop} \label{p:regularity} Given a solution $(g_t, \mu_{f_t}, h_t)$ to (\ref{f:nRYM}), suppose there exists a constant $C > 0$ so that for all $t > 0$,
\begin{align} \label{f:regassump}
C_S(g_t) \leq C, \qquad C^{-1} \leq \Vol(g_t) \leq C, \qquad \brs{\brs{d u}}_{L^2} \leq C, \qquad \FF(u_t,f_t,h_t) \leq C.
\end{align}
There exists $\ge, A$ depending on $C$ so that if $[t_0,t_1]$ is a time interval such that
\begin{align*}
\brs{t_1 - t_0} + \FF(u_{t_0},f_{t_0},h_{t_0}) - \FF(u_{t_1},f_{t_1},h_{t_1}) \leq \ge,
\end{align*}
then
\begin{align*}
\sup_{[t_0,t_1]} \left( \brs{\brs{u}}_{H^2}^2 + \brs{\brs{ e^u \brs{\N^g F_{\mu_f}}_g}}_{L^2}^2 \right) \leq A \left( \brs{\brs{u(t_0)}}_{H^2} + \brs{\brs{ e^u \brs{\N^g F_{\mu_f}}_g}(t_0)}_{L^2}^2 + 1\right).
\end{align*}
Moreover, if the assumptions (\ref{f:regassump}) hold on a finite time interval $[0,T)$, then the flow extends smoothly past time $T$.
\begin{proof} This is established in (\cite{StreetsRYMsurfaces} \S 4), and is easily modified to account for the normalizations we have chosen here.
\end{proof}
\end{prop}

\subsection{Symmetry Reductions} \label{ss:reductions}

In this subsection we record two ways in which the Ricci-Yang-Mills flow on Riemann surfaces arises via considering natural flow equations in higher dimensions with symmetries.  First, certain generalized Ricci flow lines on a three-manifold with a principal $S^1$ symmetry reduce to Ricci-Yang-Mills flow on the base space.  Also, invariant solutions to pluriclosed flow on complex surfaces which are principal $T^2$ bundles reduce to Ricci-Yang-Mills flow on the base space.  These reductions show how to obtain Corollaries \ref{c:GRFcor} and \ref{t:PCFcorollary} from Theorem \ref{t:RYMthm}, but are also instrumental in showing the behavior of the Ricci-Yang-Mills flow in the most difficult case of the sphere.

\subsubsection{$T^k$-invariant generalized Ricci flow over Riemann surfaces}

As described in the introduction, the \emph{generalized Ricci flow} is the parabolic system for a Riemannian metric $G$ and closed three-form $H$ defined by
\begin{align*}
\dt G =&\ -2 \Rc_G + \tfrac{1}{2} H^2,\\
\dt H =&\ \gD_d H.
\end{align*}
This is a parabolic system of equations and basic regularity and long-time existence obstructions have been established in \cite{Streetsexpent}.  A Perelman-type $\FF$-functional for this system was found in \cite{OSW}, and an expander entropy functional was found in \cite{Streetsexpent}.  Some recent results in the homogeneous setting have appeared \cite{ParadisoGRF}, as well as a stability result near Ricci-flat metrics \cite{VezzoniGRF}.  In \cite{GRFbook} it was shown that the equation reduces to Ricci-Yang-Mills flow in the case of a $U(1)$ principal bundle over a Riemann surface.  This extends to $T^k$ bundles:

\begin{prop} \label{p:invGRF} Let $T^k \to M \to \Sigma$ be a principal $T^k$-bundle over a Riemann surface, and suppose $(g_t, \mu_t)$ is a solution of Ricci-Yang-Mills flow (\ref{f:RYM}).  Let
\begin{align*}
G_t = \pi^* g + \tr_h \mu_t \otimes \mu_t,\quad H_t = \tr_h F_{\mu_t} \wedge \mu_t.
\end{align*}
Then $(G_t, {H}_t)$ is a solution of generalized Ricci flow.
\begin{proof} The proof is identical to (cf. \cite{GRFbook} Prop 4.39), which is written for the case $k = 1$ but generalizes immediately to the case of arbitrary $k$.
\end{proof}
\end{prop}

\subsubsection{$T^2$-invariant pluriclosed flow on complex surfaces}

Given a complex manifold $(M^{2n}, J)$, a Hermitian metric $g$ is \emph{pluriclosed} if the associated K\"ahler form $\gw$ satisfies $\i \del\delb \gw = 0$.  For a pluriclosed metric we define $H = - d^c \gw = \i (\del - \delb) \gw$, noting $d H = 0$.  There is a Hermitian connection on $TM$, the \emph{Bismut connection}, defined by $\N^B = \N + g^{-1} H$.  This has an associated curvature tensor $\Omega^B$, and the Bismut-Ricci form is the natural contraction
\begin{align*}
\rho_B = \tfrac{1}{2} \tr \Omega^B J.
\end{align*}
The \emph{pluriclosed flow} is the equation
\begin{align*}
\dt \gw =&\ - \rho_B^{1,1}.
\end{align*}
This is a parabolic equation \cite{PCF}, which solves K\"ahler-Ricci flow if the initial metric is K\"ahler.  Furthermore, after a gauge modification, the associated pairs $(g_t, H_t)$ of metrics and three-forms are a solution of generalized Ricci flow \cite{PCFReg}.  Global existence and convergence results for pluriclosed flow have appeared in \cite{LeeStreets, StreetsPCFBI, StreetsGG}.  Pluriclosed flow also preserves generalized K\"ahler geometry \cite{GKRF}, and global existence and convergence results have been shown in this setting \cite{ASnondeg, StreetsSTB, Streetsvisc, StreetsND}, and specifically the results of \cite{StreetsSTB} overlap partially with Corollary \ref{t:PCFcorollary}.

Let us now restrict to the case where $(M^4, J)$ is a complex surface which is the total space of a holomorphic principal $T^2$ bundle over a base manifold $\Sigma$.  Let $Z,W$ denote canonical vertical vector fields associated to a basis $\mathfrak Z, \mathfrak W$ for the torus action, such that $W = JZ$.  Let $g$ denote an invariant Hermitian metric on $J$, and let $\IP{,}$ denote a metric on $\mathfrak t^2$ such that $\IP{\mathfrak Z, \mathfrak W} = 0$.  As explained in \cite{Streetssolitons}, a choice of invariant Hermitian metric $G$ is equivalent to a triple $(g, \mu, \psi)$, where
\begin{align*}
\psi =&\ G(Z,Z) = G(W,W),\\
\mu(X) =&\ \psi^{-1} g(X,Z) \mathfrak Z + \psi^{-1} g(X,W) \mathfrak W,\\
	g(X,Y) =&\ G(X,Y) - \psi \IP{\mu(X), \mu(Y) }.
\end{align*}
Furthermore, the metric is pluriclosed if and only if $\psi$ is constant.  Here $\psi \IP{,}$ is playing the role of the metric $h$ as described above.  By general principle the pluriclosed flow will preserve the $T^2$ symmetry, and the fact that $\psi$ is constant.  As shown in (\cite{Streetssolitons} Lemma 6.2), the natural Hermitian connection is determined by functions $f_1$ and $f_2$ where
\begin{align*}
\mu_f = \left( \mu^{\mathfrak Z} + d^c f_1 + d f_2 \right) \otimes \mathfrak Z + \left( \mu^{\mathfrak W} + d^c f_2 - d f_1 \right) \otimes W.
\end{align*}
Furthermore, the symmetry reduced equations are gauge-equivalent to the Ricci-Yang-Mills flow:
\begin{prop} \label{p:flowreduction} (cf. \cite{Streetssolitons} Proposition 6.3) Given $(M^4, J)$ as above, a one-parameter family of $T^2$ invariant metrics $G_t = \pi^* g_t + \tr_{h_t} \mu_f \otimes \mu_f$ is a solution to normalized pluriclosed flow if and only if
\begin{gather} \label{f:reducedPCF}
\begin{split}
\frac{\del g}{\del t} =&\ -  \left( R - \psi \brs{F_{\mu_f}}^2 + \gl 
\right) g^T,\\
\dt f_1 =&\ \tr_{\gw} F^{\mathfrak Z}_{\mu_{f}} = \gD_{g} f_1 
+ \tr_{\gw} F^{\mathfrak Z}_{\mu},\\
\dt f_2 =&\ \tr_{\gw} F^{\mathfrak W}_{\mu_{f}} = \gD_{g} f_2 
+ \tr_{\gw} F^{\mathfrak W}_{\mu},\\
\dt \psi =&\ - \gl \psi.
\end{split}
\end{gather}
\end{prop}
The proof in \cite{Streetssolitons} is by direct computation, but a more conceptual proof for a more general case can be given as follows.  Suppose $T^k \to M \to \Sigma$ is a principal $T^k$ bundle.  By taking a product with a trivial $T^k$ bundle we obtain a $T^{2k}$ bundle $\til{M}$ over $\Sigma$, which admits a natural complex structure induced by a choice of complex structure on $\mathfrak t^{2k}$ determined by the natural splitting.  We can extend a choice of principal connection on $M$ using a flat connection to determine a principal connection $\til{\mu}$ on $\til{M}$.  Letting $\xi^i$ denote a basis for $T^k$, it follows that the $(1,1)$ form
\begin{align*}
\gw = \pi^* \gw_{g} + \sum_{i=1}^k \mu^{\xi_i} \wedge \mu^{J \xi_i}
\end{align*}
is positive, and moreover $H = - d^c \gw = \sum_{i=1}^k F^{\xi_i} \wedge \mu^{\xi_i}$.  Since the solution to pluriclosed flow is equivalent to generalized Ricci flow after a gauge transformation (\cite{PCFReg}), it follows from Proposition \ref{p:invGRF} that, up to a gauge transformation, the solution to pluriclosed flow is the same as the associated solution to Ricci-Yang-Mills flow with initial data $(g, \mu)$ on $\Sigma$.

\section{A priori estimates} \label{s:estimates}

In this section we establish a priori estimates building towards the global existence claims of Theorem \ref{t:RYMthm}.  In all cases we will choose a background conformal metric $g_{\Sigma}$ with constant scalar curvature.  With this choice we obtain a solution $u_t$ to the flow of conformal factors (\ref{f:ured}) by Lemma \ref{l:uevolution}.  Furthermore using Hodge theory we can choose a background connection $\bar{\mu}$ such that
\begin{align*}
F_{\bar{\mu}} \equiv \gw_{\Sigma} \otimes \zeta \in \Lambda^2 (\Sigma) \otimes \mathfrak t^k.
\end{align*}
With this choice we apply Lemma \ref{l:freduction} to obtain a solution $f_t$ to the potential flow (\ref{f:fred}).

\subsection{Evolution equations}

\begin{lemma} \label{l:invtracelemma} Given a solution to Ricci-Yang-Mills flow, we have
\begin{align} \label{f:invtraceev}
\left(\dt - \gD \right) e^{-u} =&\ - \brs{\N u}^2_{g_t} e^{-u} + R_{\Sigma} e^{-2u}  - \tfrac{1}{2} \brs{F_{\mu_f}}_{g_t,h_t}^2 e^{-u} + \gl e^{-u}.
\end{align}
\begin{proof} We directly compute using Lemma \ref{l:uevolution}
\begin{align*}
\dt e^{-u} =&\ - e^{-u} \left( \gD u - e^{-u} R_{\Sigma} + \tfrac{1}{2} \brs{F_{\mu_f}}^2_{g_t, h_t} - \gl \right)\\
=&\ \gD e^{-u} - \brs{\N u}^2_{g_t} e^{-u} + R_{\Sigma} e^{-2u}  - \tfrac{1}{2} \brs{F_{\mu_f}}_{g_t,h_t}^2 e^{-u} + \gl e^{-u}.
\end{align*}
\end{proof}
\end{lemma}

\begin{lemma} \label{l:gradflemma} Given a solution to Ricci-Yang-Mills flow, we have
\begin{gather} \label{f:est5}
\begin{split}
\left(\dt - \gD \right) \brs{\N f}^2_{g_t,h_0} =&\ - 2 \brs{\N^2 f}^2_{g_t,h_0} - \tfrac{1}{2} \brs{F_{\mu_{f}}}_{g_t,h_t}^2 \brs{\N f}^2_{g_t,h_0} + \gl \brs{\N f}^2_{g_t,h_0}\\
&\ + 2  \IP{\N e^{-u} \otimes \tr_{\gw_{\Sigma}} F_{\bar{\mu}}, \N f}_{g_t,h_0}.
\end{split}
\end{gather}
\begin{proof} We directly compute using Lemma \ref{l:uevolution}, Lemma \ref{l:freduction}, and the Bochner formula
\begin{align*}
\left(\dt - \gD \right) \brs{\N f}^2_{g_t,h_0} =&\ \left( R - \tfrac{1}{2} \brs{F_{\mu_{f}}}_{g_t,h_t}^2 + \gl \right) \brs{\N f}^2_{g_t,h_0} + 2 \IP{\N \left(\gD f + \tr_{\gw} F_{\bar{\mu}} \right), \N f}_{g_t,h_0}\\
&\ - 2 \IP{\gD \N f, \N f}_{g_t,h_0} - 2 \brs{\N^2 f}^2_{g_t,h_0}\\
=&\ - 2 \brs{\N^2 f}^2_{g_t,h_0} - \tfrac{1}{2} \brs{F_{\mu_{f}}}_{g_t,h_t}^2 \brs{\N f}^2_{g_t,h_0} + \gl \brs{\N f}^2_{g_t,h_0}\\
&\ + 2  \IP{\N e^{-u} \otimes \tr_{\gw_{\Sigma}} F_{\bar{\mu}}, \N f}_{g_t,h_0},
\end{align*}
as claimed.
\end{proof}
\end{lemma}

\subsection{The case \texorpdfstring{$\chi(\Sigma) < 0$}{}}

Assume $\chi(\Sigma) < 0$, so that by the uniformization theorem we may choose the background metric $g_{\Sigma}$ so that $R_{g_{\Sigma}} = -1$.  In this case we also set $\gl = 1$, and these choices hold throughout this subsection.

\begin{prop} \label{p:invtraceest} Given a solution to Ricci-Yang-Mills flow with $\chi(\Sigma) < 0$, we have
\begin{align*}
\sup_{M \times \{t\}} \left( e^{-u} - 1 \right) \leq&\ e^{-t} \sup_{M \times \{0\}} \left( e^{-u} - 1 \right).
\end{align*}
\begin{proof} Specializing Lemma \ref{l:invtracelemma} to the case $R_{\Sigma} = -1$, $\gl = 1$, and dropping negative terms yields
\begin{align*}
\left( \dt - \gD \right) \left( e^{-u} - 1\right) \leq&\ - (e^{-u})^2 + e^{-u}\\
=&\ - \left( e^{-u} - 1 \right) - \left(e^{-u} - 1 \right)^2.
\end{align*}
The result follows from the maximum principle.
\end{proof}
\end{prop}

\begin{prop} \label{p:fest} Given a solution to Ricci-Yang-Mills flow with $\chi(\Sigma) < 0$, we have
\begin{align*}
\sup_{M \times \{t\}} \brs{f}_{h_0} \leq C (1+t).
\end{align*}
\begin{proof} Using the a priori estimate of Proposition \ref{p:invtraceest} we see by Lemma \ref{l:freduction} that
\begin{align*}
\left(\dt - \gD \right) f =&\ \tr_{\gw} F_{\bar{\mu}} \leq C.
\end{align*}
The upper bound follows by the maximum principle, and the lower bound is similar.
\end{proof}
\end{prop}

\begin{lemma} \label{l:ubprop1} Given a solution to Ricci-Yang-Mills flow  with $\chi(\Sigma) < 0$, there exists a constant $A > 0$ so that
\begin{align*}
\left(\dt - \gD \right) & \left( A (e^{-u} - 1) + e^{-t} \brs{\N f}_{g_t,h_0}^2 + e^{-t} \brs{f}_{h_0}^2 \right)\\
\leq&\ - \brs{F_{\mu_f}}_{g_t,h_t}^2 - \frac{C^{-1} A}{2} \brs{\N e^{-u}}^2_{g_t} - e^{-t} \brs{\N f}_{g_t,h_0}^2 - \frac{A}{2} e^{-2u} + 2 A e^{-u}.
\end{align*}
\begin{proof} Fix $A > 0$ and let $\Phi = A \left( e^{-u} - 1 \right) + e^{-t} \brs{\N f}^2_{g_t,h_0} + e^{-t} |f|_{h_0}^2$.  By combining Lemmas \ref{l:freduction}, \ref{l:invtracelemma}, and \ref{l:gradflemma} we obtain
\begin{gather} \label{f:ubprop10}
\begin{split}
\left( \dt - \gD \right) \Phi =&\ A \left\{ - \brs{\N u}^2_{g_t} e^{-u} - \tfrac{1}{2} \brs{F_{\mu_f}}_{g_t,h_t}^2 e^{-u} - e^{-2u} + e^{-u} \right\}\\
&\ + e^{-t} \left\{- 2 \brs{\N^2 f}^2_{g_t,h_0} - \brs{F_{\mu_f}}_{g_t,h_t}^2 \brs{\N f}^2_{g_t,h_0} + 2 \IP{\N e^{-u} \otimes \tr_{\gw_{\Sigma}} F_{\mu} , \N f}_{g_t,h_0} \right\}\\
&\ - 2 e^{-t} \brs{\N f}^2_{g_t,h_0} + 2 e^{-t} e^{-u} \left< \tr_{\gw_{\Sigma}} F_{\bar{\mu}}, f \right>_{h_0} - e^{-t} |f|_{h_0}^2\\
\leq&\ - C^{-1} A \brs{\N e^{-u} }^2_{g_t} + A \left( e^{-u} - e^{-2u} \right)\\
&\ - 2 e^{-t} \brs{\N^2 f}^2_{g_t,h_0}  + C e^{-t} \brs{\N e^{-u}}_{g_t} \brs{\N f}_{g_t,h_0} - 2 e^{-t} \brs{\N f}^2_{g_t,h_0}\\
&\ + C e^{-t} e^{-u} \brs{f}_{h_0} - e^{-t} |f|_{h_0}^2\\
\leq&\ - 2 e^{-t} \brs{\N^2 f}^2_{g_t,h_0} - \frac{C^{-1} A}{2} \brs{\N e^{-u}}^2_{g_t} - e^{-t} \brs{\N f}^2_{g_t,h_0} + A \left( e^{-u} - e^{-2u} \right)\\
&\ + C (1 + t)e^{-t} e^{-u},
\end{split}
\end{gather}
where the second line follows from the estimate of Proposition \ref{p:invtraceest}, and the third line follows by choosing $A$ sufficiently large and applying the Cauchy-Schwarz inequality, and the estimate of Proposition \ref{p:fest}.  Next we use the formula $F_{\mu_f} = F_{\bar{\mu}} + \i \del \delb f $, together with the fact that $h_t = e^{-t} h_0$ to obtain the estimate
\begin{gather} \label{f:ubprop15}
\begin{split}
\brs{\N^2 f}^2_{g_t,h_0} \geq&\ \brs{ \i \del \delb f}^2_{g_t,h_0}\\
=&\ \brs{F_{\mu_f} - F_{\bar{\mu}}}^2_{g_t,h_0}\\
\geq&\ \tfrac{1}{2} \brs{F_{\mu_f}}^2_{g_t,h_0} - 4 \brs{F_{\bar{\mu}}}^2_{g_t,h_0}\\
\geq&\ \tfrac{1}{2} e^t \brs{F_{\mu_f}}^2_{g,h} - C e^{-2u}.
\end{split}
\end{gather}
Inserting this into (\ref{f:ubprop10}) and choosing $A$ sufficiently large yields the inequality 
\begin{align*}
\left( \dt - \gD \right) \Phi \leq&\ - \brs{F_{\mu_f}}^2_{g_t,h_t} - \frac{C^{-1} A}{2} \brs{\N e^{-u}}^2_{g_t} - e^{-t} \brs{\N f}^2_{g_t,h_0} + A \left( e^{-u} - e^{-2u} \right)\\
&\ + C (1 + t) e^{-t} e^{-u} + C (e^{-u})^2\\
\leq&\ - \brs{F_{\mu_f}}^2_{g_t,h_t} - \frac{C^{-1} A}{2} \brs{\N e^{-u}}^2_{g_t} - e^{-t} \brs{\N f}^2_{g_t,h_0} - \frac{A}{2} e^{-2u} + 2 A e^{-u},
\end{align*}
where the last line follows by choosing $A$ sufficiently large.
\end{proof}
\end{lemma}

\begin{prop} \label{p:gradfest} Given a solution to Ricci-Yang-Mills flow  with $\chi(\Sigma) < 0$, there exists a constant $C > 0$ so that
\begin{align*}
\sup_{M \times \{t\}} e^{-t} \brs{\N f}^2_{g_t,h_0} \leq&\ C.
\end{align*}
\begin{proof} Defining $\Phi$ as in Lemma \ref{l:ubprop1} and returning to line (\ref{f:ubprop10}) and using the result of Proposition \ref{p:invtraceest} yields the differential inequality
\begin{align*}
\left(\dt - \gD \right) \Phi \leq&\ - e^{-t} \brs{\N f}^2_{g_t,h_0} + C.
\end{align*}
Using the estimates of Proposition \ref{p:invtraceest} and \ref{p:fest}, at a sufficiently large maximum for $\Phi$ it follows that  $e^{-t} \brs{\N f}^2_{g_t,h_0}$ is also arbitrarily large, yielding a contradiction.  Hence $\Phi$ has an a priori upper bound and thus so does $e^{-t} \brs{\N f}^2_{g_t,h_0}$.
\end{proof}
\end{prop}

\begin{prop} \label{p:ubprop2} Given a solution to Ricci-Yang-Mills flow  with $\chi(\Sigma) < 0$, there exists a constant $C > 0$ so that
\begin{align*}
\sup_{M \times \{t\}} u \leq C.
\end{align*}
\begin{proof} Choose $A > 0$ so that the conclusion of Lemma \ref{l:ubprop1} holds, and let
\begin{align*}
\Phi = u + A ( e^{-u} - 1) + e^{-t} \brs{\N f}^2_{g_t,h_0} + e^{-t} \brs{f}_{h_0}^2
\end{align*}
Combining Lemma \ref{l:uevolution} and Lemma \ref{l:ubprop1} we obtain, dropping some negative terms,
\begin{gather} \label{f:ub210}
\begin{split}
\left(\dt - \gD \right) \Phi \leq&\ \tfrac{1}{2} \brs{F_{\mu_f}}^2_{g_t,h_t} + e^{-u} - 1 + \left\{ - \brs{F_{\mu_f}}_{g_t,h_t}^2 - \frac{A}{2} e^{-2u} + 2 A e^{-u} \right\}\\
\leq&\ C e^{-u} - 1.
\end{split}
\end{gather}
Suppose there exists $(x_0,t_0)$ such that
\begin{align*}
\Phi(x_0,t_0) = \sup_{M \times [0,t_0]} \Phi = B.
\end{align*}
Using Propositions \ref{p:invtraceest}, \ref{p:fest}, and \ref{p:gradfest}, if $B$ is sufficiently large, it follows that
\begin{align*}
u(x_0,t_0) \geq \frac{B}{2}.
\end{align*}
Rearranging this inequality we obtain
\begin{align*}
e^{-u}(x_0,t_0) \leq&\ e^{-\frac{B}{2}}.
\end{align*}
Since $(x_0,t_0)$ is a spacetime maximum for $\Phi$, returning to (\ref{f:ub210}) and applying the maximum principle we obtain
\begin{align*}
0 \leq&\ \left[ \left(\dt - \gD \right) \Phi \right](x_0,t_0) \leq C e^{-\frac{B}{2}} - 1,
\end{align*}
a contradiction for $B$ chosen sufficiently large.  It follows that $\Phi$ has a uniform upper bound, and the proposition follows.
\end{proof}
\end{prop}

\subsection{The case \texorpdfstring{$\chi(\Sigma) = 0$}{}}

Assume $\chi(\Sigma) = 0$.  By the uniformization theorem we may choose the background metric $g_{\Sigma}$ so that $R_{\Sigma} = 0$.  In this case we again fix $\gl = 1$, and these choices hold throughout this subsection.

\begin{prop} \label{p:KodLB} Given a solution to Ricci-Yang-Mills flow  with $\chi(\Sigma) = 0$, we have
\begin{align*}
\sup_{M \times \{t\}} e^{-u} \leq&\ e^{t} \sup_{M \times \{0\}} e^{-u}.
\end{align*}
\begin{proof} Since in this setting $R_{\Sigma} = 0$, this follows directly by applying the maximum principle to the evolution equation of Lemma \ref{l:invtracelemma}.
\end{proof}
\end{prop}

\begin{prop} \label{p:fest2} Given a solution to Ricci-Yang-Mills flow  with $\chi(\Sigma) = 0$, we have
\begin{align*}
\sup_{M \times \{t\}} \brs{f}_{h_0} \leq C e^t.
\end{align*}
\begin{proof} Using the a priori estimate of Proposition \ref{p:KodLB} we see by Lemma \ref{l:freduction} that
\begin{align*}
\left(\dt - \gD \right) f =&\ \tr_{\gw} F_{\bar{\mu}} \leq C e^t.
\end{align*}
The upper bound follows by the maximum principle, and the lower bound is similar.
\end{proof}
\end{prop}

\begin{lemma} \label{l:KodUBlemma} Given a solution to Ricci-Yang-Mills flow  with $\chi(\Sigma) = 0$, there exists a constant $A > 0$ so that
\begin{align*}
\left(\dt - \gD \right) &\left(A e^{-u} + e^{-t} \brs{\N f}_{g_t,h_0}^2 + e^{-t} |f|_{h_0}^2 \right)\\
\leq&\ - \brs{F_{\mu_f}}^2_{g_t,h_t} - \frac{C^{-1} A}{2} e^{-t} \brs{\N e^{-u}}^2_{g_t} - e^{-t} \brs{\N f}^2_{g_t,h_0} + C e^{-u}.
\end{align*}
\begin{proof} Fix $A > 0$ and let $\Phi = A e^{-u} + e^{-t} \brs{\N f}^2_{g_t,h_0} + e^{-t} |f|_{h_0}^2$.  By combining Lemmas \ref{l:freduction}, \ref{l:invtracelemma}, and \ref{l:gradflemma}, using $R_{\Sigma} = 0$, and the estimates of Propositions \ref{p:KodLB} and \ref{p:fest2} we obtain

\begin{gather} \label{f:Kubprop10}
\begin{split}
\left( \dt - \gD \right) \Phi =&\ A \left\{ - \brs{\N u}^2_{g_t} e^{-u} - \tfrac{1}{2} \brs{F_{\mu_f}}_{g_t,h_t}^2 e^{-u} + e^{-u} \right\}\\
&\ + e^{-t} \left\{- 2 \brs{\N^2 f}^2_{g_t,h_0} - \brs{F_{\mu_f}}_{g_t,h_t}^2 \brs{\N f}^2_{g_t,h_0} + 2 \IP{\N e^{-u} \otimes \tr_{\gw_{\Sigma}} F_{\bar{\mu}}, \N f}_{g_t,h_0} \right\}\\
&\ - 2 e^{-t} \brs{\N f}^2_{g_t,h_0} + 2 e^{-t} e^{-u} \IP{\tr_{\gw_{\Sigma}} F_{\bar{\mu}}, f}_{h_0} - e^{-t} |f|_{h_0}^2\\
\leq&\ - C^{-1} A e^{-t} \brs{\N e^{-u}}^2_{g_t} + C e^{t}\\
&\ - 2 e^{-t} \brs{\N^2 f}^2_{g_t,h_0}  + C e^{-t} \brs{\N e^{-u}}_{g_t} \brs{\N f}_{g_t,h_0} - 2 e^{-t} \brs{\N f}_{g_t,h_0}^2\\
&\ + C e^{-t}  e^{-u} |f|_{h_0} - e^{-t} |f|_{h_0}^2\\
\leq&\ - 2 e^{-t} \brs{\N^2 f}^2_{g_t,h_0} - \frac{C^{-1} A}{2} e^{-t} \brs{\N e^{-u}}^2_{g_t} - e^{-t} \brs{\N f}^2_{g_t,h_0} + C e^{-u}.
\end{split}
\end{gather}
Arguing as in line (\ref{f:ubprop15}) we obtain
\begin{gather*}
\begin{split}
\brs{\N^2 f}^2_{g_t,h_0} \geq&\ \tfrac{1}{2} e^{t} \brs{F_{\mu_f}}_{g_t,h_t}^2 - C e^{-2u}.
\end{split}
\end{gather*}
Using this estimate in (\ref{f:Kubprop10}) and applying Proposition \ref{p:KodLB}, the result follows.
\end{proof}
\end{lemma}

\begin{prop} \label{p:KodUB} Given a solution to Ricci-Yang-Mills flow  with $\chi(\Sigma) = 0$, one has
\begin{align*} 
\sup_{M \times \{t\}} u \leq&\ C e^t.
\end{align*}
\begin{proof} Choose $A > 0$ so that the conclusion of Lemma \ref{l:KodUBlemma} holds, and let
\begin{align*}
\Phi = u + A e^{-u} + e^{-t} \brs{\N f}^2_{g_t,h_0} + e^{-t} |f|_{h_0}^2.
\end{align*}
Combining Lemmas \ref{l:uevolution} and \ref{l:KodUBlemma} we obtain, dropping some negative terms and applying Proposition \ref{p:KodLB},
\begin{gather*}
\begin{split}
\left(\dt - \gD \right) \Phi \leq&\ \tfrac{1}{2} \brs{F_{\mu_f}}^2_{g_t,h_t} - 1  + \left\{ - \brs{F_{\mu_f}}_{g_t,h_t}^2 + C e^{-u}\right\}\\
\leq&\ C e^{-u} - 1\\
\leq&\ C e^t - 1.
\end{split}
\end{gather*}
Applying the maximum principle gives the result.
\end{proof}
\end{prop}

\subsection{The case \texorpdfstring{$\chi(\Sigma) > 0$, $c_1(M) = 0$}{}}

Assume $\chi(\Sigma) > 0$, so that by the uniformization theorem we may choose the background metric $g_{\Sigma}$ so that $R_{\Sigma} = 1$.  We furthermore make the assumption that $c_1(M) = 0 \in \mathfrak t^k$, and thus we may choose a background connection $\bar{\mu}$ so that
\begin{align*}
F_{\bar{\mu}} \equiv 0.
\end{align*}
We also set $\gl = 0$, noting then that $h_t \equiv h_0$.  These choices hold throughout this subsection.

\begin{lemma} \label{l:trivS2gradf} Given a solution to Ricci-Yang-Mills flow  with $\chi(\Sigma) > 0$, $c_1 = 0$, one has
\begin{align*}
\left(\dt - \gD \right) \brs{\N f}^2_{g_t,h_0} =&\ -2 \brs{F_{\mu_f}}^2_{g_t,h_0} -2 \brs{(\N^2 f)^{2,0 + 0,2}}^2_{g_t,h_0} - \brs{F_{\mu_f}}_{g_t,h_0}^2 \brs{\N f}^2_{g_t,h_0}
\end{align*}
\begin{proof} In the case $c_1 = 0$, Lemma \ref{l:gradflemma} yields
\begin{align*}
\left(\dt - \gD \right) \brs{\N f}^2_{g_t,h_0} =&\ - 2 \brs{\N^2 f}^2_{g_t,h_0} - \brs{F_{\mu_f}}_{g_t,h_0}^2 \brs{\N f}^2_{g_t,h_0}.
\end{align*}
Since $F_{\bar{\mu}} \equiv 0$, it follows that
\begin{align*}
\brs{\N^2 f}^2_{g_t,h_0} = \brs{(\N^2 f)^{1,1}}^2_{g_t,h_0} + \brs{ (\N^2 f)^{2,0 + 0,2}}^2_{g_t,h_0} = \brs{F_{\mu_f}}^2_{g_t,h_0} + \brs{(\N^2 f)^{2,0 + 0,2}}^2_{g_t,h_0},
\end{align*}
and the result folows.
\end{proof}
\end{lemma}

\begin{prop} \label{p:trivS2gradf} Given a solution to Ricci-Yang-Mills flow  with $\chi(\Sigma) > 0$, $c_1 = 0$, one has
\begin{gather}
\sup_{M \times \{t\}} \brs{\N f}^2_{g_t,h_0} \leq \sup_{M \times \{0\}} \brs{\N f}^2_{g_0,h_0}.
\end{gather}
\begin{proof} Lemma \ref{l:gradflemma} yields
\begin{align*}
\left(\dt - \gD \right) \brs{\N f}^2_{g_t,h_0} \leq&\ 0,
\end{align*}
and the result follows from the maximum principle.
\end{proof}
\end{prop}

\subsection{The case \texorpdfstring{$\chi(\Sigma) > 0$, $c_1(M) \neq 0$}{}}

\begin{prop} \label{p:S2volume} Given a solution to Ricci-Yang-Mills flow as above with $\chi(\Sigma) > 0$, $c_1 \neq 0$, there exists a constant $C > 0$ so that
\begin{align*}
C^{-1} \leq \Vol(g_t) \leq C.
\end{align*}
\begin{proof} To establish the lower bound we use the fact that the bundle is nontrivial.  Fix $X \in \mathfrak t^k$ such that $\int_{\Sigma} \IP{F_{\mu}, X} \neq 0$.  Then we observe that there is a uniform constant $\gl > 0$ such that
\begin{align*}
\gl = \brs{\int_{\Sigma} \IP{F_{\mu},X}_h} \leq \int_{\Sigma} \brs{F_{\mu}}_{g,h} dV_g \leq \brs{\brs{ \brs{F_{\mu_f}}_{g,h}}}_{L^2} \Vol(g_t)^{\tfrac{1}{2}}.
\end{align*}
It follows that
\begin{align*}
\frac{d}{dt} \Vol(g_t) \geq&\ - 4 \pi + \frac{\gl^2}{2 \Vol(g_t)},
\end{align*}
and the lower bound follows.  Using the upper bound on the Liouville energy from Proposition \ref{p:Liouvillemon}, we obtain an upper bound on the volume from the upper bound on Liouville energy and the sharp form of the Moser-Trudinger inequality on $S^2$ \cite{AubinBest, MoserSharp, Trudinger}.
\end{proof}
\end{prop}

\section{Global existence and convergence} \label{s:convergence}
\subsection{The case \texorpdfstring{$\chi(\Sigma) < 0$}{}}

\begin{prop} \label{p:RYMconvergence} Let $T^k \to M \to \Sigma$ denote a principal torus bundle over a compact Riemann surface with $\chi(\Sigma) < 0$.  Given $G_0 = \pi^*g_0 + \tr_{h_0} \mu_0 \otimes \mu_0$ an invariant metric on $M$, the solution to (\ref{f:nRYM}) with $\gl = 1$ and initial condition $(g_0,\mu_0,h_0)$ exists on $[0,\infty)$, and further satisfies
\begin{enumerate}
\item $\lim_{t \to \infty} g_t = g_{\Sigma}$, the unique conformal metric on $\Sigma$ of constant curvature $-1$.
\item $\lim_{t \to \infty} F_{\mu_t} = F_{\bar{\mu}} = \gw_{\Sigma} \otimes \zeta$ for some $\zeta \in \mathfrak{t}^k$.
\item $(M, G_t)$ converges in the Gromov-Hausdorff topology to $(\Sigma, g_{\Sigma})$.
\end{enumerate}
\begin{proof} By combining the estimates of Propositions \ref{p:invtraceest}, \ref{p:ubprop2}, and then applying Proposition \ref{p:regularity}, we conclude long time existence of the flow.  Note that the upper and lower estimates on $u$ are uniform in time.

To address the convergence we use the energy monotonicity.  Since $\gl = 1$ it follows from Proposition \ref{p:Liouvillemon}  that $\FF$ is monotonically decreasing.  As the functional $\FF$ is bounded below, it follows that there exists a sequence $\{t_i\} \to \infty$ such that
\begin{align*}
\lim_{i \to \infty} \frac{d}{dt} \FF(u_{t_i}, f_{t_i}) = 0.
\end{align*}
Choose a small $\ge > 0$ and fix a time $t_i$ such that $0 \geq \frac{d}{dt} \FF(u_{t_i},f_{t_i}) \geq - \ge$.  From Proposition \ref{p:Liouvillemon} it follows that for such times $t_i$ one has
\begin{align*}
\int_{\Sigma} e^u \dot{u}^2 dV_{\Sigma} + \int_{\Sigma} e^{-u} \brs{F_{\mu_f}}^2_{g_{\Sigma},h} dV_{\Sigma} + 2 \int_{\Sigma} e^{-2u} \IP{\N^g F_{\mu_f}, \N^g F_{\mu_f}}_{g_{\Sigma},h}dV_{\Sigma} \leq \ge.
\end{align*}
Using the uniform bound for $u$ from Propositions \ref{p:invtraceest}, \ref{p:ubprop2}, it follows that
\begin{align} \label{f:PEconv10}
\int_{\Sigma} \dot{u}^2 dV_{\Sigma} + \int_{\Sigma} \brs{F_{\mu_f}}^2_{g,h} dV_{\Sigma} + 2 \int_{\Sigma} \brs{\N^g F_{\mu_f}}_{g,h}^2 dV_{\Sigma} \leq C \ge.
\end{align}
We next show that the $L^2$ smallness of $\dot{u}$ can be split into the scalar curvature and bundle curvature pieces to give smallness of the $H^2$ norm of $u$.  To that end we estimate using the Sobolev inequality,
\begin{align*}
\int_{\Sigma} \brs{F_{\mu_f}}_{g,h}^4 dV_{\Sigma} =&\ \brs{\brs{ \brs{F_{\mu_f}}_{g,h}}}_{L^4}^4 \leq C \brs{\brs{\brs{F_{\mu_f}}_{g,h}}}_{L^2}^2 \brs{\brs{\brs{F_{\mu_f}}_{g,h}}}_{H^1}^2 \leq C \ge^2.
\end{align*}
It follows from the $L^2$ estimate for $\dot{u}$ above that
\begin{align} \label{f:PEconv20}
\int_{\Sigma} \left( e^{-u} \left(\gD_{\Sigma} u + 1 \right) - 1 \right)^2 dV_{\Sigma} \leq C \ge.
\end{align}
As $u$ is also uniformly bounded, we conclude a uniform $H^2$ estimate for $u$ at these times, and thus a uniform $C^{\ga}$ estimate by Sobolev embedding.  Applying Proposition \ref{p:regularity} we obtain a uniform $H^2$ bound for $u$ on a time interval of a definite length, and thus a uniform $C^{\ga}$ bound as well. Since times $t_i$ satisfying $-\ge \leq \frac{d}{dt} \FF(u_{t_i}, f_{t_i})$ are arbitrarily close to any sufficiently large time, it follows that there is a uniform $C^{\ga}$ bound for $u$ for all times $t \geq 0$.  Returning to (\ref{f:PEconv10}), we conclude an $H^2$ estimate for $f - \bar{f}$ at time $t_i$, and thus a $C^{\ga}$ estimate as well.  It now follows from parabolic Schauder estimates that $f - \bar{f}$ has a uniform $C^{2,\ga}$ estimate.  The equation for $u$ is now uniformly parabolic with a $C^{\ga}$ inhomogeneous term, thus we obtain a uniform $C^{2,\ga}$ estimate for $u$ on $[0,T)$.  Alternating between the equations for $f$ and $u$ we continue bootstrapping to get $C^{\infty}$ estimates for $u$ and $f - \bar{f}$ on $[t_i,t_i+1]$.  It follows that there are uniform $C^{\infty}$ estimates for $u$ and $f- \bar{f}$.

From the discussion above, specifically (\ref{f:PEconv20}), it follows that the metrics $g_t$ converge to a conformal metric of constant curvature $-1$, which must be $g_{\Sigma}$, and in particular $u$ converges to zero.  With this in place it follows from (\ref{f:fred}) that $f - \bar{f}$ converges to zero, so that $F_{\mu_f} \to F_{\bar{\mu}}$, as clamed.  As $h_t = e^{-t} h_0$, it follows immediately that $(M, G_t = \pi^* g_t + h_t \mu_{t} \otimes \mu_t)$ converges to $(\Sigma, g_{\Sigma})$ as claimed.
\end{proof}
\end{prop}

\subsection{The case \texorpdfstring{$\chi(\Sigma) = 0$}{}}

This case requires a further technical lemma to obtain the limiting behavior, namely a bound on the Sobolev constant of the conformal metric.  This lemma is also employed in obtaining estimates in the case $\Sigma = S^2$.

\begin{lemma} \label{l:confSobest} Suppose $(\Sigma, g_{\Sigma})$ is a compact Riemann surface.  Fix $u \in C^{\infty}(\Sigma)$ such that
\begin{align*}
\int_{\Sigma} u dV_{\Sigma} = 0, \qquad \brs{\brs{d u}}_{L^2} \leq A.
\end{align*}
There exists a constant $C = C(A)$ such that $C_S(e^u g_{\Sigma}) \leq C$.
\begin{proof} Using the $L^2$ gradient bound for $u$, it follows from the Moser-Trudinger inequality that for any $p > 1$ one has a constant $C$ such that
\begin{align*}
\int_{\Sigma} e^{p \brs{u}} dV_{\Sigma} \leq C.
\end{align*}
We note that $W^{1,2}$ embeds into any $L^p$ space, so we choose the constant $C_S(g_{\Sigma})$ such that
\begin{align*}
\left( \int_{\Sigma} f^8 dV_{\Sigma} \right)^{\tfrac{1}{4}} \leq C_S(g_{\Sigma}) \left( \int_{\Sigma} \brs{d f}^2_{g_{\Sigma}} dV_{\Sigma} + \int_{\Sigma} f^2 dV_{\Sigma} \right).
\end{align*}
Using the above estimates we will bound the Sobolev constant for the metric $g = e^u g_{\Sigma}$, for the embedding $W^{1,2} \to L^4$. The argument modifies in on obvious way to estimate other Sobolev constants.  To begin we have
\begin{align*}
\left(\int_{\Sigma} f^4 dV_{g} \right)^{\frac{1}{2}} =&\ \left( \int_{\Sigma} f^4 e^u dV_{\Sigma} \right)^{\frac{1}{2}}\\
\leq&\ \left( \int_{\Sigma} f^8 dV_{\Sigma} \right)^{\frac{1}{4}} \left( \int_{\Sigma} e^{2u} dV_{\Sigma} \right)^{\frac{1}{4}}\\
\leq&\ C \left( \int_{\Sigma} f^8 dV_{\Sigma} \right)^{\frac{1}{4}}.
\end{align*}
Using the Sobolev inequality for $g_{\Sigma}$ and the conformal invariance of the Dirichlet energy we obtain
\begin{align*}
\left(\int_{\Sigma} f^4 dV_{g} \right)^{\frac{1}{2}} + \left( \int_{\Sigma} f^8 dV_{\Sigma} \right)^{\frac{1}{4}} \leq&\ C  \int_{\Sigma} \left( \brs{\N f}^2_{g_{\Sigma}} + f^2 \right) dV_{\Sigma}\\
=&\ C \int_{\Sigma} \brs{\N f}^2_{g} dV_g + C \int_{\Sigma} f^2 dV_{\Sigma}.
\end{align*}
We furthermore choose $N > 0$ and observe that
\begin{align*}
\int_{e^{-u} \geq N} dV_{\Sigma} \leq \frac{1}{N} \int_{\Sigma} e^{\brs{u}} dV_{\Sigma} \leq \frac{C}{N}.
\end{align*}
Using this we estimate
\begin{align*}
\int_{\Sigma} f^2 dV_{\Sigma} =&\ \int_{e^{-u} \leq N} f^2 dV_{\Sigma} + \int_{e^{-u} \geq N} f^2 dV_{\Sigma}\\
\leq&\ N \int_{\Sigma} f^2 dV_g + \left( \int_{\Sigma} f^8 \right)^{\frac{1}{4}} \left( \int_{e^{-u} \geq N} dV_{\Sigma} \right)^{\frac{3}{4}}\\
\leq&\ N \int_{\Sigma} f^2 dV_g + C N^{-\frac{3}{4}} \left( \int_{\Sigma} f^8 dV_{\Sigma} \right)^{\frac{1}{4}}.
\end{align*}
Choosing $N$ sufficiently large and using this above gives the claim.
\end{proof}
\end{lemma}

\begin{prop} \label{p:RYMconvergenceg1} Let $T^k \to M \to \Sigma$ denote a principal torus bundle over a compact Riemann surface with $\chi(\Sigma) = 0$.  
Given $G_0 = \pi^*g_0 + \tr_{h_0} \mu_0 \otimes \mu_0$ an invariant metric on $M$, the solution to (\ref{f:nRYM}) with $\gl = 1$ and initial condition $(g_0,\mu_0,h_0)$ exists on $[0,\infty)$ and $(M, G_t)$ converges in the Gromov-Hausdorff topology to a point.
\begin{proof} By combining the estimates of Propositions \ref{p:KodLB}, \ref{p:KodUB}, and then applying Proposition \ref{p:regularity}, we conclude long time existence of the flow.  In this case we cannot conclude $L^{\infty}$ estimates for $u$ which are uniform in time.

To obtain the convergence, we first observe that by Proposition \ref{p:Liouvillemon}, since $\gl = 1$, there is an a priori upper bound on $\FF$, which in this case immediately implies a uniform upper bound for $\brs{\brs{\N u}}_{L^2}$.  Furthermore, using that $\FF$ is bounded below since $R_{\Sigma} = 0$ we choose a sequence of times $\{t_i\} \to \infty$ such that
\begin{align*}
\lim_{i \to \infty} \frac{d}{dt} \FF(u_{t_i}, f_{t_i}) = 0.
\end{align*}
Thus we may fix $\ge > 0$ arbitrary and then for sufficiently large $i$ it follows from Proposition \ref{p:Liouvillemon} that for $t_i$ one has
\begin{align} \label{f:g110}
\int_{\Sigma} \left[ e^u \dot{u}^2 + e^{-u} \brs{F_{\mu_f}}^2_{g_{\Sigma},h} + 2 e^{-2u} \IP{\N^g F_{\mu_f}, \N^g F_{\mu_f}}_{g_{\Sigma},h} \right] dV_{\Sigma} \leq \ge.
\end{align}
Note that we can decompose the first term to yield
\begin{align*}
\ge \geq&\ \int_{\Sigma} e^u \left( - R_g + \tfrac{1}{2} \brs{F_{\mu}}^2_{g,h} - 1 \right)^2 dV_{\Sigma}\\
=&\ \int_{\Sigma} e^u \left( - R_g + \tfrac{1}{2} \brs{F_{\mu}}^2_{g,h} \right)^2 dV_{\Sigma} + \int_{\Sigma} \left( R_g - \tfrac{1}{2} \brs{F_{\mu}}^2_{g,h} \right) e^u dV_{\Sigma} + \int_{\Sigma} e^u dV_{\Sigma}\\
\geq&\ \int_{\Sigma} e^u \left( - R_g + \tfrac{1}{2} \brs{F_{\mu}}^2_{g,h} \right)^2 dV_{\Sigma} - 2 \ge + \int_{\Sigma} e^u dV_{\Sigma}.
\end{align*}
Thus
\begin{align} \label{f:g120}
\int_{\Sigma} e^u dV_{\Sigma} + \int_{\Sigma} e^u \left( e^{-u} \gD_{\Sigma} u + \tfrac{1}{2} \brs{F_{\mu}}^2_{g,h} \right)^2 dV_{\Sigma}. \leq \ge.
\end{align}
Now let $w = u_{t_i} - \int_{\Sigma} u_{t_i} dV_{\Sigma}$.  Using Lemma \ref{l:confSobest} we conclude that the $W^{1,2} \to L^4$ Sobolev constant of $g = e^{w} g_{\Sigma}$ is bounded.  It follows from a standard iteration argument that there is a uniform lower bound on the volume of balls of the form $\Vol(B_r(p)) \geq c r^4$.  Since the volume is uniformly bounded above it follows from this that the diameter of the metric $e^{w} g_{\Sigma}$ is bounded.  Since
\begin{align*}
\int_{\Sigma} u_{t_i} dV_{\Sigma} \leq \log \left(\int_{\Sigma} e^{u_{t_i}} dV_{\Sigma} \right) \leq \log \ge,
\end{align*}
it follows easily that the diameter of $e^{u_{t_i}} g$ is approaching zero as $i \to \infty$.

We note using the monotonicity and lower bound for $\FF$ that for any $\ge > 0$ there exists a time $T > 0$ such that for all $t \geq T$, there exists $\til{t} \in [t, t + 1]$ such that $\frac{d}{dt} \FF (u_{\til{t}}, f_{\til{t}}) \leq \ge$.  The argument above applies at time $\til{t}$ to show that $\diam (g_{\til{t}}) = o(\ge)$.  We further claim that for all $t \geq T$, $\diam(g_t) = o(\ge)$.  To see this, fix a time $t \geq T$ and choose $\til{t} \in [t, t+1]$ as above.  Since $H^1(\Sigma) \neq 0$, we apply the diameter lower bound (with normalization $\gl = 1$) for Ricci-Yang-Mills flow (\cite{StreetsRGflow, Knopfdiam}) to see that there is a uniform constant $c > 0$ so that $\diam(g_t) \leq c \diam (g_{\til{t}})$.  The claim follows, and we conclude $\lim_{t \to \infty} \diam(g_t) = 0$.  Since $h_t = e^{-t} h_0$, it follows that $\lim_{t \to \infty} \diam (G_t) = 0$, and so $(M, G_t)$ converges in the Gromov-Hausdorff topology to a point.
\end{proof}
\end{prop}

\subsection{The case \texorpdfstring{$\chi(\Sigma) > 0$}{} and trivial bundle}
\subsubsection{\texorpdfstring{$\gk$}{}-noncollapsing} \label{ss:knonc}

Here we sketch an argument for proving uniform $\gk$-noncollapsing of the metrics in this setting, adapting arguments of (\cite{StreetsThesis, GindiStreets, GRFbook}), all based on the original argument of Perelman \cite{Perelman1}.  To begin we define a modification of Perelman's entropy formula adapted to Ricci-Yang-Mills flow, which is not quite monotone, though still of some interest in studying Ricci-Yang-Mills flow more generally.  Similar entropy monotonicity formulas have appeared in \cite{Lott1}.  In the setting of a trivial bundle, we can make a further modification which yields a monotone quantity which gives the required $\gk$-noncollapsing.

\begin{defn} Given $T^k \to M \to \Sigma$ a principal $T^k$ bundle over $\Sigma$, fix data $(g,\mu,h)$ as above, with $F_{\mu} = d \mu \in \Lambda^2 T^* \otimes \mathfrak t^k$.  Fix $u \in C^{\infty}(M), u > 0$, $\tau > 0$, and define $f_-$ via $u = \frac{e^{-f_-}}{(4 \pi \tau)^{\frac{n}{2}}}$.  Define
\begin{align*}
\WW_-(g,\mu,h,f_-,\tau) =&\ \int_{\Sigma} \left[ \tau \left( \brs{\N f_-}^2 + R - \frac{1}{4} \brs{F_{\mu}}^2_{g,h} \right) + f_- - n \right] u dV.
\end{align*}
\end{defn}

\begin{lemma} \label{l:entvar} Given $T^k \to M \to \Sigma$ a principal $T^k$ bundle over $\Sigma$, fix data $(g,\mu,h, u, \tau)$ as above, and let
\begin{align*}
\gd g = v_g, \quad \gd h = v_h, \quad \gd \mu = \ga, \quad \gd f = \phi, \quad \gd \tau = \gs.
\end{align*}
Then
\begin{align*}
\gd \WW_- & \left( v_g, v_h, \ga, \phi, \gs \right)\\
=&\ \int_{\Sigma} \left[ \gs \left(\brs{\N f_-}^2 + R - \frac{1}{4} \brs{F_{\mu}}^2_{g,h} \right) - \tau \left<v_g, \Rc - \frac{1}{2} F_{\mu}^2 + \N^2 f \right> - \frac{\tau}{4} \IP{v_h, \tr_g F_{\mu} \otimes F_{\mu}} \right. \\
&\ \qquad \left. - \tau \IP{ \ga, d^*_g F_{\mu} + \N f_- \lrcorner F_{\mu}} + \phi \right.\\
&\ \qquad \left. + \left[ \tau \left(2 \gD f_- - \brs{\N f_-}^2 + R - \frac{1}{4} \brs{F_{\mu}}^2_{g,h} \right) + f - n \right] \left(\frac{\tr_g v_g}{2} - \phi - \frac{n \gs}{2 \tau} \right) \right] u dV_g
\end{align*}
\end{lemma}

\begin{defn} Given $T^k \to M \to \Sigma$ a principal $T^k$ bundle over $\Sigma$, suppose $(g_t, \mu_t)$ is a solution of Ricci-Yang-Mills flow.  We say that a one-parameter family of functions $f_-$ satisfies the \emph{conjugate heat equation} if
\begin{align*}
\dt f_- = - \gD f_- + \brs{\N f_-}^2 - R + \frac{1}{2} \brs{F}^2_{g,h} + \frac{n}{2 \tau}.
\end{align*}
\end{defn}

\begin{prop} \label{p:entropymon1} Let $(M^n, g_t, \mu_t, h)$ be a solution to Ricci-Yang-Mills flow.  Let $f_-$ denote an associated solution of the conjugate heat equation.  Then, for $T > 0$, setting $\tau = T - t$,
\begin{align*}
\frac{d}{dt} \WW_- \left( g, \mu, h, f_-, \tau \right) =&\ \int_{\Sigma} \left[ 2 \tau \brs{\Rc - \frac{1}{2} F_{\mu}^2 + \N^2 f_- - \frac{g}{2 \tau}}^2 + \tau \brs{d^*_g F_{\mu} + \N f_- \lrcorner F_{\mu}}^2 - \frac{3}{4} \brs{F_{\mu}}^2_{g,h} \right] u dV_g.
\end{align*}
\begin{proof} By diffeomorphism invariance of the functional it suffices to compute for the gauge-fixed system:
\begin{align*}
\dt g =&\ -2 \Rc + F^2_{\mu} - 2 \N^2 f_-,\\
\dt \mu =&\ - d^*_g F_{\mu} - \N f_-  \lrcorner F_{\mu},\\
\dt f_- =&\ - \gD f_- - R + \frac{1}{2} \brs{F_{\mu}}^2_{g,h} + \frac{n}{2 \tau}.
\end{align*}
Observe that for this system of equations, $\frac{1}{2} \tr_g \dt g - \dt f_- + \frac{n}{2 \tau} = 0$.  It follows from Lemma \ref{l:entvar} that
\begin{align*}
\frac{d}{dt} \WW_- & \left( g, \mu, h, f_-, \tau \right)\\
=&\ \int_{\Sigma} \left[ 2 \tau \left< \Rc - \frac{1}{2} F_{\mu}^2 + \N^2 f_-, \Rc - \frac{1}{2} F_{\mu}^2 + \N^2 f_- \right> + \tau \brs{d^*_g F_{\mu} + \N f_- \lrcorner F_{\mu}}^2 \right.\\
&\ \qquad \left.  - \gD f_- - R + \frac{1}{2} \brs{F_{\mu}}^2_{g,h} + \frac{n}{2 \tau} - \left( \brs{\N f_-}^2 + R - \frac{1}{4} \brs{F_{\mu}}^2_{g,h} \right) \right] u dV_g\\
=&\ \int_{\Sigma} \left[ 2 \tau \brs{\Rc - \frac{1}{2} F_{\mu}^2 + \N^2 f_- - \frac{g}{2 \tau}}^2 + \tau \brs{d^*_g F_{\mu} + \N f_- \lrcorner F_{\mu}}^2 - \frac{3}{4} \brs{F_{\mu}}^2_{g,h} \right] u dV_g,
\end{align*}
as claimed.
\end{proof}
\end{prop}

While the result of Proposition \ref{p:entropymon1} is suggestive, the presence of the one negative term prevents immediate application as a monotonicity formula with attendant estimates.  This can be rectified in the presence of a bounded subsolution of the heat equation with inhomogeneous term $-\frac{3}{4} \brs{F}^2_{g,h}$ (cf. \cite{GRFbook} Ch. 6 where one employs a `torsion-bounding subsolution').  In the setting of a trivial torus bundle over a Riemann surface, this is provided by $\brs{\N f}^2_{g_t, h_0}$ using Lemma \ref{l:trivS2gradf}.

\begin{prop} \label{p:entropymon2} Let $T^k \to M \to \Sigma$ be the trivial $T^k$ bundle over a Riemann surface $\Sigma$.  Let $(M^n, g_t, \mu_t, h)$ be a solution to Ricci-Yang-Mills flow, with associated potential functions $f_t$.  Let $f_-$ denote an associated solution of the conjugate heat equation.  Then, for $T > 0$,setting $\tau = T - t$,
\begin{align*}
\frac{d}{dt} & \left[ \WW_- \left( g, \mu, h, f_-, \tau \right) - \int_{\Sigma} \brs{\N f}^2_{g_t,h_0} u dV_g \right]\\
&\ \qquad \geq \int_{\Sigma} \left[ 2 \tau \brs{\Rc - \frac{1}{2} F^2 + \N^2 f_- - \frac{g}{2 \tau}}^2 + \tau \brs{d^*_g F + \N f_- \lrcorner F}^2  \right] u dV_g\\
&\ \qquad \geq 0.
\end{align*}
 \begin{proof} Using that $\dt \left( u dV_g \right) = - \gD u dV_g$, it follows from Lemma \ref{l:trivS2gradf} that
 \begin{align*}
 \frac{d}{dt} \int_{\Sigma} \brs{\N f}^2_{g_t,h_0} u dV_g =&\ \int_{\Sigma} \left( \dt - \gD \right) \brs{\N f}^2_{g_t, h_0} u dV_g\\
 =&\ \int_{\Sigma} \left( -2 \brs{F_{\mu_f}}^2 -2 \brs{(\N^2 f)^{2,0 + 0,2}}^2 - \brs{F_{\mu_f}}_{g,h}^2 \brs{\N f}^2_{g_t,h_0} \right) u dV_g.
 \end{align*}
 Combining this with Proposition \ref{p:entropymon1} gives the result.
 \end{proof}
\end{prop}

Using Proposition \ref{p:entropymon2} it is possible to obtain a uniform $\gk$-noncollapsing result for Ricci-Yang-Mills flow in this setting, as well as the existence of a blowup limit at any finite time singularity which is uniformly $\gk$-noncollapsed (cf. \cite{Perelman1} for the definitions)

\begin{thm} \label{t:blowups} Let $T^k \to M \to \Sigma$ be the trivial $T^k$ bundle over a Riemann surface $\Sigma$.  Let $(M^n, g_t, \mu_t, h)$ be a solution to Ricci-Yang-Mills flow.  Then $g_t$ is not locally collapsing at $T < \infty$.  Furthermore, assume $(x_i, t_i)$ satisfies
\begin{enumerate}
\item $\gL_i := \brs{\Rm}(x_i,t_i) \to \infty$,
\item $\sup_{M \times [0,t_i]} \brs{\Rm}(x_i,t_i) \leq C \gL_i$,
\end{enumerate}
then the sequence of pointed solutions $\left\{ (\gL_i g_i(t_i + \gL_i^{-1} t), \mu(t_i + \gL_i^{-1} t), \gL_i h(t_i + \gL_i^{-1}t) \right\}$ converges subsequentially to a complete ancient solution to Ricci-Yang-Mills flow which is $\gk$-noncollapsed on all scales for some $\gk > 0$.
\begin{proof} The proof of no-local-collapsing follows the original argument of Perelman using the monotonicity formula of Proposition \ref{p:entropymon2} (cf. \cite{GRFbook} Ch. 6 for the modifications coming from the inclusion of the $\brs{\N f}^2_{g_t,h_0}$ term).  Using this, the construction of the blowup limit follows from the compactness theory for Ricci-Yang-Mills flow/generalized Ricci flow (\cite{StreetsThesis}, \cite{GRFbook} Ch. 5).
\end{proof}
\end{thm}

\begin{rmk} The blowup limits constructed in Theorem \ref{t:blowups} will necessarily be solutions to Ricci-Yang-Mills flow on $\mathbb R^k$-bundles, since the fiber metric is fixed along the flow, and thus blowing up along the rescaled sequence.
\end{rmk}

\subsubsection{Global existence and convergence}

\begin{prop} \label{p:trivialS2conv} Let $T^k \to \Sigma \times T^k \to \Sigma$ denote the trivial principal $T^k$-bundle over a compact Riemann surface with $\chi(\Sigma) > 0$.  Given $G_0 = \pi^* g_0 + \tr_{h_0} \mu_0 \otimes \mu_0$ an invariant metric on $\Sigma \times T^k$, the solution to (\ref{f:nRYM}) with $\gl = 0$ and initial condition $(g_0,\mu_0,h_0)$ exists on a finite time interval $[0,T)$, and satisfies
\begin{enumerate}
\item $\lim_{t \to T} (T - t) \brs{F_{\mu_t}}^2_{g_t,h_0} = 0$,
\item $\lim_{t \to T} (T - t)^{-1} g_t = g_{\Sigma}$,
\end{enumerate}
$g_{\Sigma}$ denotes a metric of constant curvature $1$.
\begin{proof} By taking a double cover if necessary we can reduce to the case $\Sigma \cong S^2$.  We first claim that the existence time is finite.  We describe the proof in the case $k = 1$ for simplicity, the general case is directly analogous.  As in the proof of Proposition \ref{p:Hopfconvergence2}, we can extend the principal connection $\mu_0$ using a flat connection to obtain a Hermitian connection $\underline{\mu}_0$ on the trivial $T^2$ bundle on $S^2 \times T^2$.  This defines a choice of complex structure on $T^2$, which then yields a complex structure on the product $S^2 \times T^2$.  Furthermore, by (\cite{Streetssolitons} Proposition 5.3) the data $(g_0, \underline{\mu}_0, h_0)$ defines a Hermitian metric on $S^2 \times T^2$, which is furthermore pluriclosed by (\cite{Streetssolitons} Lemma 5.6).  By (\cite{Streetssolitons} Proposition 6.3, cf. also Remark 6.4) it follows that the pluriclosed flow with this initial data will reduce to Ricci-Yang-Mills flow.  The $S^2$ slices are all holomorphic curves in this complex manifold, and it follows from Stokes Theorem (cf. \cite{PCF} Proposition 3.9) that for the solution $\gw_t$ to pluriclosed flow one has
\begin{align*}
\frac{d}{dt} \int_{S^2} \gw_t = - 8 \pi.
\end{align*}
It follows that the area goes to zero in finite time for any initial data, and thus the flow must go singular at a time $T < (8\pi)^{-1} \Area_{\gw_0}(S^2)$.

Fix a solution $(g_t, \mu_t)$ to Ricci-Yang-Mills flow as in the statement, with associated potential functions $f_t$.  If the flow exists on a maximal time interval $[0, T)$, $T < \infty$, by an elementary point-picking argument we can choose a sequence $(x_i, t_i)$ of points satisfying conditions (1) and (2) of Theorem \ref{t:blowups}, and construct a corresponding blowup limit $(g^{\infty}_t, \mu^{\infty}_t, h^{\infty})$ as described in Theorem \ref{t:blowups}.  We note that by construction, the sequence of functions $f^i(x,t) = f(x, t_i + \gL_i^{-1} t)$ are potential functions for the rescaled flows, and the $f^i$ will converge to a limit function $f^{\infty}$ satisfying $F_{\mu^{\infty}_t} = d d^c f^{\infty}_t$.  On the other hand, since $\brs{\N f}^2_{g_t, h_0}$ has a uniform upper bound by Proposition \ref{p:trivS2gradf}, it follows by the scaling law that $\brs{\N f^{\infty}}^2_{g^{\infty}, h_0} \equiv 0$.  Thus $f^{\infty}$ is constant, and so $F_{\mu^{\infty}} \equiv 0$.  It follows that $g^{\infty}_t$ is an ancient $2$-dimensional $\gk$-solution to Ricci flow, which by (\cite{Perelman1} Corollary 11.3) is isometric to the standard round shrinking sphere solution.  Using this blowup sequence it is straightforward to show that the whole flow line converges to a round point and satisfies the scale-invariant decay of curvature claimed in the statement.
\end{proof}
\end{prop}

\subsection{The case \texorpdfstring{$\chi(\Sigma) > 0$}{}, nontrivial bundle}

\begin{prop} \label{p:Hopfconvergence2} Let $T^k \to P \to \Sigma$ denote a principal $T^k$-bundle over a compact Riemann surface with $\chi(\Sigma) > 0$, $c_1(P) \neq 0$.  Given $G_0 = \pi^* g_0 + \tr_{h_0} \mu_0 \otimes \mu_0$ an invariant metric on $P$, the solution to (\ref{f:nRYM}) with $\gl = 0$ and initial condition $(g_0,\mu_0,h_0)$ exists on $[0,\infty)$, and there is a one-parameter family of diffeomorphisms $\phi_t$ such that
\begin{enumerate}
\item $\lim_{t \to \infty} \phi_t^* g_t = \gg_2 g_{\Sigma}$,
\item $\lim_{t \to \infty} \phi_t^* F_{\mu_t} = \gg_1 \gw_{\Sigma}$,
\end{enumerate}
where $g_{\Sigma}$ denotes a metric of constant curvature $1$ and $\gw_{\Sigma}$ denotes the associated area form.
\begin{proof} First, by lifting to a double cover it suffices to consider the case $\Sigma = S^2$.  We first employ a gauge-fixing procedure (cf. \cite{StruweFlows}) to account for the noncompact group of conformal diffeomorphisms of $S^2$.  For such a diffeomorphism $\phi$ we will set
\begin{align*}
\phi^* g = \phi^* \left( e^u g_{S^2} \right) = e^{u \circ \phi} \phi^* g_{\Sigma} = e^{v} g_{\Sigma} = \hat{g},
\end{align*}
where
\begin{align} \label{f:vdef}
v = u \circ \phi + \log \det d \phi.
\end{align}
The condition we require, aiming at application of the sharp Sobolev inequality of Aubin \cite{AubinBest}, is that
\begin{align} \label{f:gauge}
\int_{S^2} x dV_{\hat{g}} = 0,
\end{align}
where $x$ denotes the position vector in $\mathbb R^3$.  By (\cite{ChangMTI} Lemma 2) we can find a conformal transformation $\phi_0$ so that (\ref{f:gauge}) is satisfied at time $t = 0$. To solve for the relevant gauge transformations we first set $\xi_t = \left(\phi_t^{-1}\right)_* \dot{\phi}$, which is a vector field on $S^2$.  Differentiating equation (\ref{f:vdef}) gives
\begin{align*}
\dot{v} = \dot{u} \circ \phi + \xi v + \divg_{g_{S^2}} \xi.
\end{align*}
Differentiating the defining condition (\ref{f:gauge}) yields (cf. \cite{StruweFlows} Lemma 6.2)
\begin{align*}
0 =&\ \int_{S^2} \left[ x \left(  -2 v e^{-v} + \tfrac{1}{2} \brs{F_{\mu_f}}^2_{\hat{g},h} \right) - \xi \right] dV_{\hat{g}}.
\end{align*}
As explained in (\cite{StruweFlows} \S 5.2) this condition suffices to solve for $\xi$.  Thus we obtain the required family of diffeomorphisms $\phi_t$ such that
\begin{align*}
\phi_t^* g_t = \phi_t^* (e^u g_{S^2}) = e^{v_t} g_{S^2} = \hat{g}_t,
\end{align*}
where condition (\ref{f:gauge}) is satisfied at all times.  These diffeomorphisms define new curvature densities $\hat{F}_t = \phi_t^* F_t$, realized as the curvature of a family of connections $\hat{\mu}_t$ satisfying
\begin{align*}
\dt \hat{\mu}_t = - d^*_{\hat{g}} F_{\hat{\mu}} + i_{\xi_t} F_{\hat{\mu}}.
\end{align*}

Next we observe the key qualitative influence of the hypothesis that $c_1(P) \neq 0$, namely the a priori volume lower bound provided by Lemma \ref{p:S2volume}.  It follows that $\hat{g}_t$ also has a uniform lower volume bound.  It follows from \cite{AubinBest} that for any $\ge > 0$ we have $C = C(\ge)$ such that
\begin{align*}
c \leq \fint_{S^2} e^v dV_{S^2} \leq C \exp \left[ \left( \tfrac{1}{8} + \ge \right) \fint_{S^2} \brs{d v}^2_{g_{S^2}} dV_{S^2} + \tfrac{1}{2} R_{\Sigma} \fint_{S^2} v dV_{S^2} \right].
\end{align*}
Choosing for instance $\ge = \tfrac{1}{16}$, we can rearrange this to obtain the estimate
\begin{align} \label{f:Liobnd}
\frac{6}{16} \fint_{S^2} \brs{d v}^2_{g_{S^2}} dV_{S^2} + R_{\Sigma} \fint_{S^2} v dV_{S^2} \geq - C.
\end{align}
We also note that the functional $\FF$ is bounded above along the ungauged flow by Proposition \ref{p:Liouvillemon}.  Since the Liouville energy is invariant under the action of conformal diffeomorphisms by (\cite{ChangMTI} Lemma 1), it follows that the Liouville energy of $v$ is uniformly bounded above.  Using this together with (\ref{f:Liobnd}) gives
\begin{align*}
\brs{\brs{dv}}_{L^2} \leq C.
\end{align*}
Using this, it follows that $e^{\brs{v}}$ is uniformly bounded in any $L^p$ space, also also that $v$ is bounded in $L^1$.  Thus by Lemma \ref{l:confSobest} we obtain a uniform estimate of the Sobolev constant of $e^{v} g$.  Applying Proposition \ref{p:regularity} we obtain the long time existence.

To obtain the convergence we further analyze the monotonicity of $\FF$.  Using the Sobolev constant estimate for $\hat{g}$ and the uniform $L^2(\hat{g})$ estimate for $F_{\mu_f}$ it follows that that
\begin{align} \label{f:S2nontriv10}
\brs{\brs{ \brs{F_{\hat{\mu}}}_{g,h}}}_{L^4(\hat{g})}^2 \leq C \left( \brs{\brs{ \N^{\hat{g}} F_{\hat{\mu}}}}_{L^2(\hat{g})}^2 + \brs{\brs{F_{\hat{\mu}}}}_{L^2(\hat{g})}^2 \right) \leq C \left( \brs{\brs{ \N^{\hat{g}} F_{\hat{\mu}}}}_{L^2(\hat{g})}^2 + 1 \right),
\end{align}
where the last line follows from the upper bound for $\FF$.  
In particular, since $\FF$ is bounded below along the flow, there is a sequence $t_i \to \infty$ such that $\lim_{i \to \infty} \frac{d}{dt} \FF(v_t,\hat{\mu}_t,h) = 0$, thus, for sufficiently large $i$, at any time $t_i$ we have
\begin{align} \label{f:S2nontriv20}
\ge \geq \int_{\Sigma} \left( R_{\hat{g}} - \tfrac{1}{2} \brs{F_{\hat{\mu}}}^2_{\hat{g},h} \right)^2 dV_{\hat{g}} + 2 \int_{\Sigma} \brs{\N^{\hat{g}} F_{\hat{\mu}}}^2_{{\hat{g}},h} dV_{\hat{g}}.
\end{align}
This implies that $\brs{\brs{\N^{\hat{g}} F_{\hat{\mu}}}}_{L^2(g,h)}$ is bounded, thus applying (\ref{f:S2nontriv10}) we obtain a uniform $L^4({\hat{g}})$ estimate for $F_{\hat{\mu}}$, and it follows easily then from  (\ref{f:S2nontriv20}) that the Calabi energy of $\hat{g}$ is bounded.  Since $\int_M e^{2 \brs{v}} dV_{\Sigma} < C$, it follows that the curvature does not concentrate in $L^1$, and thus by (\cite{StruweFlows} Theorem 3.2, cf. also \cite{Chenweaklimits}), $v$ is uniformly bounded in $H^2$.  By following the arguments of Proposition \ref{p:regularity}, we obtain a uniform $H^2$ estimate for $v$ on $[t_i,t_i+1]$.  Since there exists a relevant time $t_i$ in every time interval of the form $[T, T + 1]$ for all sufficiently large $T$, it follows that there is a uniform $H^2$ estimate for $v$.  Using this and a bootstrapping argument as described in Proposition \ref{p:RYMconvergence}, it follows that there are uniform $C^{k,\ga}$ estimates for $v$ and $\hat{F}$.  Returning to the monotonicity formula for $\FF$, it follows that $\N^{\hat{g}} \hat{F} \to 0$, and $R_{\hat{g}} \to \mbox{const}$.  Thus $\hat{g}$ is converging to a round metric, and $\hat{F}$ is converging to a multiple of the area form, finishing the proof.
\end{proof}
\end{prop}

\subsection{Proofs of Main Theorems}

\begin{proof}[Proof of Theorem \ref{t:RYMthm}] The individual cases are proved in Propositions \ref{p:RYMconvergence}, \ref{p:RYMconvergenceg1}, \ref{p:trivialS2conv} and \ref{p:Hopfconvergence2}.
\end{proof}

\begin{proof}[Proof of Corollary \ref{c:GRFcor}] This follows from Proposition \ref{p:invGRF} and Theorem \ref{t:RYMthm}.
\end{proof}

\begin{proof}[Proof of Corollary \ref{t:PCFcorollary}] This follows from Proposition \ref{p:flowreduction} and Theorem \ref{t:RYMthm}.  In the case of the Hopf surface, it is easy to check that the limiting structure in fact has vanishing Bismut-Ricci tensor.  In dimension $4$ it is known (\cite{GauduchonIvanov}, cf. \cite{GRFbook} Theorem 8.26) that the metric is either Calabi-Yau or isometric to the Hopf metric. Since the Hopf surface does not admit K\"ahler metrics, it must be the Hopf metric, as claimed.
\end{proof}

\bibliographystyle{plain}

\end{document}